\documentclass[11pt]{article}

\usepackage{a4wide}
\usepackage{fancyheadings}
\usepackage{amssymb}
\usepackage{graphicx}
\usepackage{subfigure}

\usepackage{graphics,color}         
\usepackage{epsfig}
\usepackage{graphicx}
\usepackage[all]{xy}
\usepackage{pb-diagram,pb-xy}  		
\usepackage{amsgen}
\usepackage{amsmath}
\usepackage{amstext}
\usepackage{amsbsy}
\usepackage{amsopn}
\usepackage{amsfonts}               
\usepackage{amssymb}
\PassOptionsToPackage{linktocpage}{hyperref}
\usepackage[a4paper,ps2pdf]{hyperref}      
\usepackage{algorithmic_PA}  
\usepackage{algorithm_PA} 
\usepackage{verbatim}   
\usepackage{scrtime}    

\index{mapping|see{function}}
\index{map|see{function}}

\newcommand{\citeyear}{\cite}

\def \STpn {\mathbb{R}_*^{n\times p}}

\def \rrnp {\rr^{n\times p}}

\def \colsymb {\mathrm{col}}
\def \col {\colsymb}

\newcommand{\lift}[2][]{{\overline{#2}_{#1}}}

\def \nn {\mathbb{N}}

\def \rr {\mathbb{R}}

\def \calE {{\cal{E}}}

\def \calH {{\cal{H}}}

\def \calM {{\cal{M}}}

\def \calS {{\cal{S}}}

\def \calV {{\cal{V}}}

\def \setsym {{\cal S}_\mathrm{sym}}
\def \setskew {{\cal S}_\mathrm{skew}}

\def \sym {\mathrm{sym}}
\def \skew {\mathrm{skew}}

\def \trace {\mathrm{tr}}

\def \D {\mathrm{D}}

\def \d {\mathrm{d}}

\def \J {\mathrm{J}}  

\def \O {\mathrm{O}}  

\newtheorem{thrm}{Theorem}[section]

\newtheorem{prpstn}[thrm]{Proposition}

{\end{list}}

{\end{list}}

\def \Pp {P^\mathrm{p}}
\def \Pa {P^\mathrm{a}}
\def \Ps {P^\mathrm{s}}
\def \Ph {P^\mathrm{h}}
\def \im {\mathrm{im}}
\def \calJ {{\cal{J}}}

\newenvironment{proof}{{\it Proof.}}{\hspace{\stretch{1}} $\square$}

\newcommand{\pacomm}[1]{}
\newcommand{\patodo}[1]{}

\begin{document}

\title{A geometric Newton method for Oja's vector field\footnotemark[1]}
\author{P.-A. Absil\footnotemark[2] \and 
M. Ishteva\footnotemark[3] \and L. De Lathauwer\footnotemark[4]
\footnotemark[3] \and S. Van Huffel\footnotemark[3]
}
\date{\today}

\maketitle

\renewcommand{\thefootnote}{\fnsymbol{footnote}} \footnotetext[1]{This
paper presents research results of the Belgian Network DYSCO
(Dynamical Systems, Control, and Optimization), funded by the
Interuniversity Attraction Poles Programme, initiated by the Belgian
State, Science Policy Office. The scientific responsibility rests with
its authors. Research supported by: (1) Research Council K.U.Leuven:
GOA-Ambiorics, CoE EF/05/006 Optimization in Engineering (OPTEC),
CIF1, (2) F.W.O.: (a) project G.0321.06, (b) Research Communities
ICCoS, ANMMM and MLDM, (3) the Belgian Federal Science Policy
Office: IUAP P6/04 (DYSCO, ``Dynamical systems, control and
optimization'', 2007--2011), (4) EU: ERNSI. M. Ishteva is supported by a K.U.Leuven doctoral scholarship (OE/06/25, OE/07/17), L. De Lathauwer
is supported by ``Impulsfinanciering Campus Kortrijk
(2007-2012)(CIF1)''.}   
\footnotetext[2]{D\'epartement
d'ing\'enierie math\'ematique, Universit\'e catholique de Louvain,
Belgium (www.inma.ucl.ac.be/$\sim$absil).}
\footnotetext[3]{Research Division SCD of the Department of Electrical Engineering
  (ESAT) of the Katholieke Universiteit Leuven, Kasteelpark Arenberg
  10, B-3001 Leuven, Belgium.
}
\footnotetext[4]{Subfaculty Science and
  Technology of the Katholieke Universiteit Leuven Campus Kortrijk,
  E. Sabbelaan 53, 8500 Kortrijk, Belgium.}
\renewcommand{\thefootnote}{\arabic{footnote}}

\thispagestyle{fancy}
\rhead{Tech.\ Report UCL-INMA-2008.013}
\lhead{07 Apr 2008}

\begin{center}
{\bf Abstract}\\[.5em]

\begin{minipage}{12cm}
\noindent
Newton's method for solving the matrix equation $F(X)\equiv
AX-XX^TAX=0$ runs up against the fact that its zeros are not
isolated. This is due to a symmetry of $F$ by the action of the
orthogonal group. We show how differential-geometric techniques can be
exploited to remove this symmetry and obtain a ``geometric'' Newton
algorithm that finds the zeros of $F$. The geometric Newton method
does not suffer from the degeneracy issue that stands in the way of
the original Newton method.
\end{minipage}

\end{center}


\section{Introduction}
\label{sec:intro}

Let $A$ be a symmetric positive-definite $n\times n$ matrix and let
$p$ be a positive integer smaller than $n$. 

Oja's flow~\cite{Oja82,Oja89}, in its averaged version
\begin{subequations}  \label{eq:Oja}
\begin{equation}  \label{eq:Oja-dyn}
\frac{\d}{\d t} X(t) = F(X(t)),
\end{equation}
where $F$ denotes the vector field
\begin{equation}  \label{eq:Oja-F}
F: \rr^{n\times p} \to \rr^{n\times p}: X \mapsto AX - XX^TAX,
\end{equation}
\end{subequations}
is well known for its principal component analysis properties. For all
initial conditions $X_0$, the ordinary differential
equation~\eqref{eq:Oja} has a unique solution curve $t\mapsto X(t)$ on
the interval $[0,\infty)$~\cite[Th.~2.1]{YHM94}, the limit $X(\infty)
= \lim_{t\to\infty}X(t)$ exists, the convergence to $X(\infty)$ is
exponential, and $X(\infty)$ is a zero of Oja's vector field
$F$~\eqref{eq:Oja-F}~\cite[Th.~3.1]{YHM94}.

Observe that the zeros
of $F$ are the solutions $X\in\rr^{n\times p}$
of the matrix equation
\[
AX = X X^TAX,
\]
which implies that the columns of $AX$ are linear combinations of the
columns of $X$. Letting
\[
\col(Y) = \{Y\alpha: \alpha\in\rr^p\}
\]
denote the column space of $Y\in\rr^{n\times p}$, it follows that all
zeros $X$ of~\eqref{eq:Oja-F} satisfy $A\,\col(X)\subseteq\col(X)$,
i.e., $\col(X)$ is an invariant subspace of $A$. Moreover, $X(\infty)$
is an orthonormal matrix ($X^T(\infty)X(\infty)=I_p$), thus of full
rank, whenever the initial condition $X(0)$ is of full
rank~\cite[Prop.~3.1]{YHM94}. Assuming that $X(0)$ has full rank, it
holds that $\col(X(\infty))$ is the $p$-dimensional principal subspace
of $A$ if and only if $\col(X(0))$ does not contain any direction
orthogonal to that subspace~\cite[Th.~5.1]{YHM94}. The set of all
initial conditions that do not satisfy this condition is a negligible
set, i.e., Oja's flow asymptotically computes the
  $p$-dimensional principal subspace of $A$ for almost all initial
  conditions. We also point out that Oja's flow induces a subspace
  flow, called the power flow~\cite{AbsMahSep2007-32}.

Because of these remarkable properties, Oja's flow has been and
remains the subject of much attention, in its form~\eqref{eq:Oja} as
well as in several modified versions; see, for
example,~\cite{YHM94,Yan1998,DMV99,ManHelMar2005,JanOga2006} .
However, turning Oja's flow into a principal component algorithm
implementable in a digital computer requires to discretize the flow in
such a way that its good convergence properties are
preserved. \patodo{Refer to Berzal and Zufira? They follow a
  stochastic approach...} Since the solutions $X(t)$ converge
exponentially to their limit point $X(\infty)$, it follows that the
sequence of equally spaced discrete-time samples $(X(kT))_{k\in\nn}$
converges only $Q$-linearly~\cite{OrtRhe1970} to $X(\infty)$. Therefore, numerical
integration methods that try to compute accurately the solution
trajectories of Oja's flow are not expected to converge faster than
linearly.

Nevertheless, if the ultimate goal is to compute the principal
eigenspace of $A$, then it is tempting to try to accelerate the
convergence when the iterates get close to the limit point
$X(\infty)$, using techniques akin to those proposed
in~\cite{Hig1999,FowKel2005,KelQiLia+2006,LuoKelLia+2006}, in order to
obtain a superlinear algorithm. To this end, it is interesting to
investigate how superlinearly convergent methods perform for finding a
zero of Oja's vector field~\eqref{eq:Oja-F}.
Since Newton's method can be thought of as the prototype superlinear
algorithm to which all other superlinear algorithms relate, we propose
to investigate the behavior of Newton's method applied to the problem
of finding a zero of Oja's vector field~\eqref{eq:Oja-F}.

A crucial hypothesis in the classical local convergence analysis of
Newton's method (see, e.g.,~\cite{DS83}) is that the targeted zero is
nondegenerate. As it turns out, the zeros of Oja's vector field $F$~\eqref{eq:Oja-F}
are never isolated, because $F$ displays a property of symmetry under
the action of the orthogonal group $\O_p$ on the set
$\rrnp$. Therefore, the classical superlinear convergence result of
Newton's method is void in the case of $F$, and indeed our numerical
experiments show that Newton's method on $\rrnp$ for $F$ performs
poorly (see Figure~\ref{fig:plainN}).

In this paper, we propose a remedy to this difficulty, that consists
in ``quotienting out'' the symmetry of $F$. Conceptually speaking,
instead of performing a Newton iteration in $\rrnp$, we perform a
Newton iteration on a quotient space, namely $\STpn/\O_p$ (defined in
Section~\ref{sec:quot}). We exploit the
Riemannian quotient manifold structure of $\STpn/\O_p$ to formulate a
Newton method on this set, following the framework developed
in~\cite{AMS2008}. 
The resulting Newton equation is a linear matrix equation that can be
solved by various numerical approaches. It follows from the theory of
Newton method on manifolds (see, e.g.,~\cite{ADM2002,AMS2008}), and from a careful analysis of the zeros
of the vector field, that the obtained algorithm converges locally
superlinearly to the set of orthonormal bases of invariant
subspaces of $A$. Our numerical experiments show that the method
behaves as expected.

\section{Plain Newton method for Oja's vector field}
\label{sec:plain}

We assume throughout the paper that $A$ is a real symmetric positive-definite
$n\times n$ matrix. For simplicity, we also assume that the
eigenvalues of $A$ satisfy 
\begin{equation}  \label{eq:simple-eig}
\lambda_1>\cdots>\lambda_n,
\end{equation}
i.e., all the eigenvalues of $A$ are simple. Hence the $p$-dimensional
invariant subspaces of $A$ are isolated.

Newton's method in $\rrnp$ for Oja's vector field $F$~\eqref{eq:Oja-F}
consists of iterating the map $X\mapsto X_+$ defined by solving
\begin{gather}
\D F(X)[Z] \equiv AZ - ZX^TAX - XZ^TAX - XX^TAZ = -F(X)
\\
X_+ = X+Z.
\end{gather}

The following proposition is a well-known characterization of the
zeros of $F$.
\begin{prpstn}[zeros of $F$~\eqref{eq:Oja-F}]
Let $X\in\rr^{n\times p}$ be of full rank. Then the following two
conditions are equivalent:
\begin{enumerate}
\item $F(X) = 0$, i.e., 
\begin{equation}  \label{eq:AX=}
AX = XX^TAX.
\end{equation}
\item $\col(X)$ is an invariant subspace of $A$ and $X$ is orthonormal
  ($X^TX=I$).
\end{enumerate}
\end{prpstn}
\begin{proof}
  $1\Rightarrow 2$. We have that $AX = X(X^TAX)$, thus $A\,\col(X)
  \subseteq \col(X)$. (More precisely, since $X$ has full rank and $A$
  is positive-definite, it
  follows that $X^TAX$ is invertible and thus $A\,\col(X) =
  \col(X)$.) Moreover, multiplying~\eqref{eq:AX=} by $X^T$ on the left
  yields $X^TAX = X^TXX^TAX$, which implies that $X^TX=I$ since
  $X^TAX$ is invertible.

$2\Rightarrow 1$. Since $\col(X)$ is an invariant subspace of $A$,
there is a matrix $M$ such that $AX = XM$. Multiply this equation on
the left by $X^T$ to obtain that $M = X^TAX$ and thus~\eqref{eq:AX=}. 
\end{proof}

Hence, the set of all full-rank zeros of $F$ is the finite union of the
compact sets
\begin{equation}  \label{eq:Si}
\calS_i := \{X\in\rr^{n\times p}: \col(X) = \calE_i, X^TX = I\}
\end{equation}
where $\calE_1,\ldots,\calE_N$ are the $p$-dimensional invariant
subspaces of $A$. (Note that $N$ is finite; it is equal to $n\choose
p$.) 
It is readily checked that $\calS_i = V_i \O_p$, where 
\[
\O_p := \{Q\in\rr^{p\times p}:Q^TQ=I\}
\]
is the orthogonal group of order $p$ and $V_i$ is an element of
$\calS_i$.  \pacomm{(One can conclude that $\calS_i$ is a submanifold
  of $\rrnp$ of the same dimension ($p(p-1)/2$) as $\O_p$.)}  It
follows that the zeros of $F$ are not isolated. Hence the Jacobian of
$F$ at the zeros of $F$ is singular (i.e., the zeros are
\emph{degenerate}), and consequently, the classical result
(see~\cite[Th.~5.2.1]{DS83}) of local superlinear convergence of
Newton's iteration to the nondegenerate zeros of $F$ is void. This
does not imply that Newton's method will fail, but there is a
suspicion that it will behave poorly, and indeed, in
Section~\ref{sec:exper} (see Figure~\ref{fig:plainN}), we report on
numerical experiments showing that this is the case. 

\section{Newton's method on $\STpn/\O_p$}
\label{sec:quot}

The degeneracy of the zeros of $F$ is due to the following fundamental
symmetry property:
\begin{equation}  \label{eq:FXQ}
F(XQ) = F(X) Q, \ \textrm{for all $Q\in \O_p$}.
\end{equation}
In this section, we propose a \emph{geometric} Newton method that
performs well with functions $F$ that satisfy~\eqref{eq:FXQ}. This
geometric Newton method evolves on the quotient space $\STpn/\O_p$,
where this symmetry is removed. Then, in Section~\ref{sec:N-Oja}, we
will return
to the specific case where $F$ is Oja's vector field~\eqref{eq:Oja-F}
and obtain a concrete numerical algorithm.

The general idea for the geometric Newton method is first to define a
vector field $\xi$ on the manifold $\STpn/\O_p$, whose zeros relate to
those of $F$. The vector field $\xi$ is specified in terms of its
so-called \emph{horizontal lift} in $\STpn$. This formulation requires
us to introduce some concepts (vertical and horizontal spaces)
borrowed from the theory of fiber bundles~\cite{KN63}, or more
specifically from the theory of Riemannian submersions~\cite{ONe1983}.
\pacomm{(The following is useless here, because the Jacobian in the
  total space is too complicated. In particular, the metric
  $\overline{g}$ is not horizontally invariant.) This theory,
  moreover, offers a convening formula for the geometric Jacobian of
  $\xi$ as the Jacobian of the horizontal lift followed by an
  orthogonal projection onto the horizontal space.)}  All the
differential-geometric concepts used in this section are explained
in~\cite{AMS2008}, or in~\cite{Lee2003} for what concerns Lie groups. 
\patodo{Announce that
  details will be elsewhere?}

The following notation will come useful. Let 
\begin{equation}  \label{eq:STpn}
\STpn = \{X\in\rrnp: \det(X^TX)\neq 0\}
\end{equation}
denote the set of all full-rank $n\times p$ matrices. Let
\begin{equation}  \label{eq:sym}
\sym(B) = \frac{1}{2}(B+B^T)
\end{equation}
and
\begin{equation}   \label{eq:skew}
\skew(B) = \frac{1}{2}(B-B^T)
\end{equation}
denote the terms of the decomposition of a square matrix $B$ into a
symmetric term and a skew-symmetric term. For $X\in\STpn$,
define 
\begin{gather}
\Pp_X: \rrnp\to\rrnp: Z\mapsto \Pp_X(Z) = (I-X(X^TX)^{-1}X^T)Z  \label{eq:Pp}
\\
\Ps_X: \rrnp\to\rrnp: Z\mapsto \Ps_X(Z) = 
X\sym((X^TX)^{-1}X^TZ)  \label{eq:Ps}
\\ 
\Pa_X: \rrnp\to\rrnp: Z\mapsto \Pa_X(Z) = 
X\skew((X^TX)^{-1}X^TZ) \label{eq:Pa}
\\ \Ph_X = \Pp_X + \Ps_X: Z\mapsto Z-X\skew((X^TX)^{-1}X^TZ).  \label{eq:Ph}
\end{gather}
We have $Z=\Pp_X(Z) + \Ps_X(Z) + \Pa_X(Z)$ for all
$Z\in\rrnp$. Observe that $\im(\Pp_X) = \{Z\in\rrnp:X^TZ=0\}$,
$\im(\Ps_X) = X\setsym(p)$, $\im(\Pa_X)=X\setskew(p)$, where
\[
\setsym(p) = \{S\in\rr^{p\times p}:S^T=S\}
\]
denotes the set of all symmetric matrices of order $p$ and
\[
\setskew(p) = \{\Omega\in\rr^{p\times p}: \Omega^T = -\Omega\}
\]
is the set of all \emph{skew-symmetric} (or \emph{antisymmetric})
matrices of order $p$. (The
letter ``p'' stands for ``perpendicular'', ``s'' for symmetric, ``a''
for antisymmetric, and the notation $\Ph$ will make sense in a moment.)

In $\STpn$, we
define an equivalence relation ``$\sim$'' where $X\sim Y$ if and only if there
exists a $Q\in \O_p$ such that $Y=XQ$. The equivalence class of
$X\in\STpn$ is thus
\begin{equation}  \label{eq:[X]}
[X] = X\, \O_p = \{XQ:Q\in \O_p\}.
\end{equation}
We let 
\[
\STpn/\O_p = \STpn / \sim
\]
denote the quotient of $\STpn$ by this equivalence
relation, i.e., the elements of the set $\STpn/\O_p$ are the
equivalence classes of the form~\eqref{eq:[X]}, $X\in\STpn$. 
We let 
\begin{equation}  \label{eq:pi}
\pi: \STpn \to \STpn/\O_p
\end{equation}
denote the \emph{quotient map} that sends $X\in\STpn$ to its
equivalence class $[X]$ viewed as an element of $\STpn/\O_p$. The set
$\STpn$ is termed the \emph{total space} of the quotient
$\STpn/\O_p$. Note that a point $\pi(X)\in\STpn/\O_p$ can be
numerically represented by any element of its equivalence class
$[X]=\pi^{-1}(\pi(X))$.

Since
$\STpn$ is an open subset of $\rrnp$, it is naturally an open
submanifold of the linear manifold $\rrnp$. Moreover, it can be shown
that the map
\[
\psi: \O_n\times \STpn\to\STpn: (Q,X)\mapsto XQ
\]
is a free and proper Lie group action on the manifold
$\STpn$. Therefore, by the quotient manifold theorem (see,
e.g.,~\cite[Th.~9.16]{Lee2003}), the orbit space $\STpn/\O_p$ is a
quotient manifold. In other words, the set $\STpn/\O_p$ is turned into
a manifold by endowing it with the unique differentiable structure
that makes the quotient map $\pi$ a submersion. It
comes as a consequence that each equivalence class
$[X]=\pi^{-1}(\pi(X))$, $X\in\STpn$, is an embedded submanifold of
$\STpn$. We term \emph{vertical space} at $X\in\STpn$ the tangent
space $\calV_X$ to $[X]$ at $X$, i.e.,
\[
\calV_X = T_X[X] = X \ T_I \O_p = X \setskew(p).
\]
Observe that
$\im(\Pa_X)=\calV_X$. 

Next we define horizontal spaces $\calH_X$, which must satisfy the
condition that $\rrnp$ is the internal direct sum of the vertical and
horizontal spaces. We choose
\begin{equation}  \label{eq:H}
\calH_X = \im(\Ph_X) = \{XS + X_\perp K: S^T=S\}
\end{equation}
where $\Ph$ is as in~\eqref{eq:Ph} and $X_\perp\in\rr^{n\times(n-p)}$
denotes any orthonormal matrix such that $X^T X_\perp=0$. (It
would be sufficient to state that $X_\perp$ has full rank, but it does
not hurt to assume that it is orthonormal. We do not use $X_\perp$ in
the numerical algorithms.)  

As an aside, we point out that the horizontal space~\eqref{eq:H} is
the orthogonal complement of the vertical space with respect to the
noncanonical metric $\overline{g}$ defined by
\[
\overline{g}_X(Z_1,Z_2) = \trace(Z_1^T \Pp_X(Z_2) + Z_1^T X(X^TX)^{-2}
X^TZ_2).
\]
Indeed, we have $\{Z\in\rrnp: \overline{g}_X(W,Z)=0 \ \text{for all
  $W\in\calV_X$} \} = \{XS + X_\perp K: S^T=S\}$. The reason for not
choosing the canonical Riemannian metric on $\STpn$ is that it yields
a more intricate formula for the horizontal space and for the
projection onto it. 

\pacomm{(We don't need the following, because the Riemannian connection on
  $\STpn$ generated by $\overline{g}$ seems to be quite complicated,
  so we won't use it, and thus we won't get the Riemannian
  connection on the quotient.)  It can be shown, using the
  developments proposed in~\cite{AMS2008}, that there exists a
  Riemannian metric $g$ on the quotient $\STpn/\O_p$ that turns the
  natural projection $\pi:\STpn\to\STpn/\O_p$ into a Riemannian
  submersion. In the language of~\cite{AMS2008}, the quotient manifold
  $\STpn/\O_p$ has thus been endowed with a Riemannian quotient
  manifold structure.  }

Consider a function $F:\rrnp\to\rrnp$ that satisfies the symmetry
property~\eqref{eq:FXQ}. Define a horizontal vector field
$\lift{\xi}$ by 
\[
\lift[X]{\xi} := \Ph_X(F(X)).
\]
It can be checked that this horizontal vector field $\lift{\xi}$
satisfies $\D \pi(X)[\lift[X]{\xi}] = \D \pi(XQ)[\lift[XQ]{\xi}]$ for
all $Q\in \O_p$, and thus $\lift{\xi}$ is the horizontal lift of a
vector field $\xi$ on the quotient manifold $\STpn/\O_p$. In the
remainder of this section, we formulate a geometric Newton method for
finding a zero of the vector field $\xi$.

\pacomm{To show that $\lift{\xi}$ indeed defines a horizontal lift,
  show that $\lift[XQ]{\xi} = \lift[X]{\xi}Q$, and that therefore
  $\lift[XQ]{\xi}$ and $\lift[X]{\xi}$ project to the same tangent
  vector at $\pi(X)$.}

For finding a zero of a vector field $\xi$ on an abstract manifold
$\calM$ endowed with an \emph{affine connection} $\nabla$ and a
\emph{retraction} $R$, we consider the geometric Newton method in the
form proposed by Shub~\cite[\S 5]{ADM2002} (or see~\cite[\S
6.1]{AMS2008}). The method consists of iterating the mapping that
sends $x\in\calM$ to $x_+\in\calM$ obtained by solving
\begin{subequations}  \label{eq:N}
\begin{gather}
\J_\xi(x)[\eta_x] = -\xi_x, \quad \eta_x\in T_x\calM,  \label{eq:N-eq}
\\
x_+ = R_x(\eta_x).  \label{eq:N-up}
\end{gather}
\end{subequations}
Here $\J_\xi(x)$ denotes the Riemannian Jacobian
$\J_\xi(x):T_x\calM\to T_x\calM: \zeta_x \mapsto \J_\xi(x)[\zeta_x] =
\nabla_{\zeta_x}\xi$.

\patodo{Define tangent space?}

For the case where $\calM$ is the quotient $\STpn/\O_p$, we choose the
affine connection $\nabla$ defined by
\begin{equation}  \label{eq:nabla}
\lift[X]{(\nabla_{\eta_{\pi(X)}}\xi)} := \Ph_X \D \lift{\xi}(X)[\lift[X]{\eta}],
\end{equation}
where the overline denotes the horizontal lift. It can be shown
that the right-hand side is indeed a horizontal lift and that the definition
of $\nabla$ satisfies all the properties of an affine connection. With
this choice for $\nabla$, and with a simple choice for the retraction
$R$ (see~\cite[\S 4.1.2]{AMS2008} for the relevant theory), the
geometric Newton method on the quotient manifold $\STpn/\O_p$ becomes
the iteration that maps $\pi(X)$ to $\pi(X_+)$ by solving
\begin{subequations}  \label{eq:N-quot}
\begin{gather}  
\Ph_X \D \lift{\xi}(X)[\lift[X]{\eta}] = -\lift[X]{\xi}, \quad
\lift[X]{\eta}\in \calH_X
\label{eq:N-quot-eq}
\\
X_+ = X+\lift[X]{\eta},
\end{gather}
\end{subequations}
with $\Ph$ as in~\eqref{eq:Ph} and $\calH_X$ as in~\eqref{eq:H}. Note
that, in spite of possibly unfamiliar notation,~\eqref{eq:N-quot} only
involves basic calculus of functions between matrix spaces.

\pacomm{A key element in the proof that we have a horizontal lift
  in~\eqref{eq:nabla} 
  is the following development:
$\D \lift{\xi}(XQ)[ZQ] = \D (Y\mapsto \lift{\xi}(Y))(XQ) \circ
\D(X\mapsto XQ)(X)[Z] = \D(X\mapsto\lift{\xi}(XQ))(X)[Z] =
\D(X\mapsto\lift{\xi}Q)(X)[Z] = \D(X\mapsto\lift{\xi}(X))(X)[Z]Q$.
}

The local superlinear convergence result for the geometric Newton
method~\eqref{eq:N-quot} follows directly from the local convergence
result of the general geometric Newton method~\eqref{eq:N},
see~\cite[\S 6.3]{AMS2008}. It states that the iteration converges
locally quadratically to the nondegenerate zeros of $\xi$. Since we
have ``quotiented out'' the symmetry of $F$ by the action of $\O_p$ to
obtain $\xi$, it is now
reasonable to hope that the zeros of $\xi$ are nondegenerate.

Note that the proposed geometric Newton method differs from the
``canonical'' Riemannian Newton algorithm~\cite{Smi94} on the quotient $\STpn/\O_p$,
because the affine connection chosen is not the Riemannian connection
on $\STpn/\O_p$ endowed with the Riemannian metric inherited from the
canonical metric in $\STpn$. However, the property of local
superlinear convergence to the nondegenerate zeros still holds
(see~\cite{ADM2002} or~\cite{AMS2008}).

\patodo{(Define $\D$?)}

\section{A geometric Newton method for Oja's vector field}
\label{sec:N-Oja}

In this section, we apply the geometric Newton method on $\STpn/\O_p$,
given in~\eqref{eq:N-quot}, to the case where the tangent vector field
$\xi$ on $\STpn/\O_p$ is defined by the horizontal lift
\begin{equation}  \label{eq:xi-Oja}
\lift[X]{\xi} := \Ph_X(F(X)),
\end{equation}
with $\Ph$ as in~\eqref{eq:Ph} and $F$ as in~\eqref{eq:Oja-F}. 
\pacomm{(It is
readily checked that~\eqref{eq:xi-Oja} is a well-defined horizontal
lift.)} The resulting Newton iteration is formulated in
Algorithm~\ref{al:N-Oja}. Recall the definitions of
$\STpn$~\eqref{eq:STpn}, $\Ph$~\eqref{eq:Ph}, $\skew$~\eqref{eq:skew},
$\calH_X$~\eqref{eq:H}.

\begin{algorithm}
\caption{Geometric Newton for Oja's vector field}
\label{al:N-Oja}
\begin{algorithmic}[1]
\REQUIRE Symmetric positive-definite $n\times n$ matrix $A$;
  positive integer $p<n$.  
\INPUT Initial iterate $X_0\in\STpn$, i.e.,
  $X_0$ is a real $n\times p$ matrix with full rank.  
\OUTPUT Sequence of
  iterates $(X_k)\subset\STpn$.  
\FOR{$k=0,1,2,\ldots$} 
\STATE Solve
  the linear system of equations (we drop the subscript $k$ for convenience)
\begin{multline}  \label{eq:N-Oja}
\Ph_X( 
AZ-ZX^TAX-XZ^TAX-XX^TAZ - Z\skew((X^TX)^{-1}X^TAX) 
\\ -
X\skew(-(X^TX)^{-1}(X^TZ+Z^TX)(X^TX)^{-1}X^TAX+(X^TX)^{-1}(Z^TAX+X^TAZ))
) 
\\ = -\left(AX-XX^TAX-X\skew((X^TX)^{-1}X^TAX)\right)
\end{multline}
for the unknown $Z\in\calH_X$.
\STATE Set
\[
X_{k+1} = X_k + Z.
\]
\ENDFOR
\end{algorithmic}
\end{algorithm}

Observe that Algorithm~\ref{al:N-Oja} is stated as an iteration
  in the total space $\STpn$ of the quotient $\STpn/\O_p$. Formally,
  the sequence of iterates of the Newton method on $\STpn/\O_p$, for
  an initial point $\pi(X_0)\in\STpn/\O_p$, is given by
  $(\pi(X_k))_{k\in\nn}$, where $\pi$ is the quotient
  map~\eqref{eq:pi} and $(X_k)_{k\in\nn}$ is the sequence of iterates
  generated by Algorithm~\ref{al:N-Oja}.

We point out that~\eqref{eq:N-Oja} is merely a linear system of
equations. It can be solved by means of iterative solvers that can
handle linear systems in operator form. Moreover, these solvers can
be stopped early to avoid unnecessary computational effort when the
iterate $X_k$ is still far away from the solution. Guidelines for stopping
the linear system solver can be found, e.g., in~\cite{EW96}. In our
numerical experiments (Section~\ref{sec:exper}) we have used
{\sc Matlab}'s GMRES solver.

Algorithm~\ref{al:N-Oja} converges locally quadratically to the
nondegenerate zeros of $\xi$. We first characterize the zeros of
$\xi$, then we show that they are all nondegenerate under the
assumption~\eqref{eq:simple-eig}.

First note that $\xi_{\pi(X)}=0$ if and only if $\Ph_X(F(X))=0$, where
$X\in\STpn$. Under the assumption that $X\in\STpn$, the following
statements are equivalent:
\begin{gather*}
\Ph_X(F(X))=0,
\\
F(X) \in \im(\Pa_X),
\\
F(X) = X\Omega \ \text{for some $\Omega=-\Omega^T$},
\\
AX -XX^TAX = X\Omega \ \text{for some $\Omega=-\Omega^T$},
\\
A\,\col(X)\subseteq\col(X) \text{ and } (X^TX)^{-1}X^TAX - X^TAX \text{
  is skew-symmetric},
\\ 
A\,\col(X)\subseteq\col(X) \text{ and } (X^TX)^{-1}X^TAX +
X^TAX(X^TX)^{-1} = 2X^TAX.
\end{gather*}
We thus have an equation of the form $YB+BY=2B$, where $B:=X^TAX$ is
symmetric positive definite, hence its eigenvalues $\beta_i$,
$i=1,\ldots,p$, are all real and positive. The Sylvester operator
$Y\mapsto YB+BY$ is a linear operator whose eigenvalues are
$\beta_i+\beta_j$, $i=1,\ldots,p$,
$j=1,\ldots,p$~\cite[Ch.~VI]{Gan59}. All these eigenvalues are real
and positive, thus nonzero, hence the operator is nonsingular, from
which it follows that the equation $YB+BY=2B$ has one and only one
solution $Y$. It is readily checked that this unique solution is $Y=I$. We
have thus shown that $X^TX=I$. The result is summarized in the
following proposition.
\begin{prpstn}
  Let $\Ph$ be as in~\eqref{eq:Ph} and let $F$ be Oja's vector
  field~\eqref{eq:Oja-F}. Then $X\in\STpn$ is a zero of the projected
  Oja vector field $\Ph(F)$ if and only if $A\,\col(X)\subseteq\col(X)$
  and $X^TX=I$.
\end{prpstn}
In other words, the full-rank zeros of the projected Oja vector field
$\Ph(F)$ are the full-rank zeros of Oja's vector field $F$. This means
that we do not lose information by choosing to search for the zeros of
$\xi$~\eqref{eq:xi-Oja} instead of $F$~\eqref{eq:Oja-F}.

It remains to show that the zeros of $\xi$ are nondegenerate. Let
$\pi(X_*)$ be a zero of $\xi$, which means that $AX_* = X_*X_*^TAX_*$
and $X_*^TX_*=I$. The task is to show that the Jacobian operator
$\J_\xi(\pi(X_*))$, or equivalently its lifted counterpart
\begin{equation}  \label{eq:N-op}
\lift{\J_\xi}(X_*): \im(\Ph)\to\im(\Ph): Z \mapsto \Ph_{X_*} \left(
\D(\Ph(F))(X_*)[Z] \right),
\end{equation}
is nonsingular. To this end, consider the operator
\begin{equation}  \label{eq:J}
\calJ:\rrnp\to\rrnp: Z\mapsto \D(\Ph(F))(X_*)[Z].
\end{equation}
Note that $\lift{\J_\xi}(X_*)$ is the restriction of $\calJ$ to
$\im(\Ph)$. Consider the decomposition
$\rrnp=\im(\Pp)\oplus\im(\Ps)\oplus\im(\Pa)$ and recall that $\im(\Ph)
= \im(\Pp)\oplus\im(\Ps)$. We show that the corresponding block
decomposition of $\calJ$ is as follows:
\[
\begin{bmatrix}
* & 0 &  0
\\ ? & * & 0
\\ ? & ? & 0
\end{bmatrix},
\]
where ``*'' denotes nonsingular operators. It then directly follows
that the upper two-by-two block of $\calJ$, which corresponds to
$\lift{\J_\xi}(X_*)$, is nonsingular.

We show that the 11 block (i.e., the ``pp'' block) is
nonsingular. This block is the operator from $\im(\Pp_{X_*})$ to
$\im(\Pp_{X_*})$ given by 
\[
Z \ \mapsto\ \Pp_{X_*} \D(\Ph(F))(X_*)[Z] = \Pp_{X_*}AZ-\Pp_{X_*} Z
X_*^TAX_*.
\]
(In obtaining this result, we have used the relations $X_*^TX=I$,
$\skew(X_*^TAX_*)=0$, $\Pp_{X_*}X_*=0$.)
In view of the hypothesis that the eigenvalues of $A$ are all
simple, this operator is known to be nonsingular; see,
e.g.,~\cite{LE2002,ASVM2004-01}.

We show that the 22 block (i.e., the ``ss'' block) is
nonsingular. This block is the operator from $\im(\Ps_{X_*})$ to
$\im(\Ps_{X_*})$ 
\[
X_*S
\quad \mapsto \quad \Ps_{X_*}\left(
\D(\Ph(F))(X_*)[X_*S] \right) = -X_*(SX_*^TAX_*+X_*^TAX_*S).
\]
(We have used the relation $AX_*=X_*X_*^TAX_*$ to obtain this
expression.) In view of the previous discussion on the Sylvester
operator, this operator is nonsingular.

This completes the proof that the zeros of $\xi$ are
nondegenerate. Consequently, for all $X_0$ sufficiently close to some
$\calS_i$~\eqref{eq:Si}, the sequence $(X_k)$ generated by
Algorithm~\ref{al:N-Oja} is such that $X_k\O_p$ converges quadratically
to $\calS_i$. Recall that $\calS_i$ is the set of all orthonormal
matrices whose column space is the $i$th invariant subspace of $A$.



\section{Numerical experiments}
\label{sec:exper}

In this section, we report on numerical experiments for both the
plain Newton and the geometric Newton method,
derived in Section 2 and Section 4, respectively. 
The experiments were run using {\sc Matlab}. The
  machine epsilon is approximately $2\cdot 10^{-16}$.

As mentioned in Section 1, 
the plain Newton method performs poorly due to the fact that the
zeros of the cost function $F$ are not isolated. To illustrate
this, we consider a symmetric positive definite matrix
$A\in\rr^{6\times 6}$ with uniformly distributed eigenvalues in the interval $[0,1]$ and 100 different initial iterates
$X_0\in \rr^{6\times 3}.$ Each of the initial iterates is computed as
$$ X_0 = X_\ast + 10^{-6} E\;,$$
where $X_\ast$ is such that $F(X_\ast)=0$ and $E$ is a $6\times 3$
matrix with random entries, chosen from a normal distribution with
zero mean and unit variance. The simulation is stopped after
50 iterations. One representative example is given in Figure~\ref{fig:plainN}.
\begin{figure}[h]
\subfigure[F]{\includegraphics[width=0.5\textwidth]{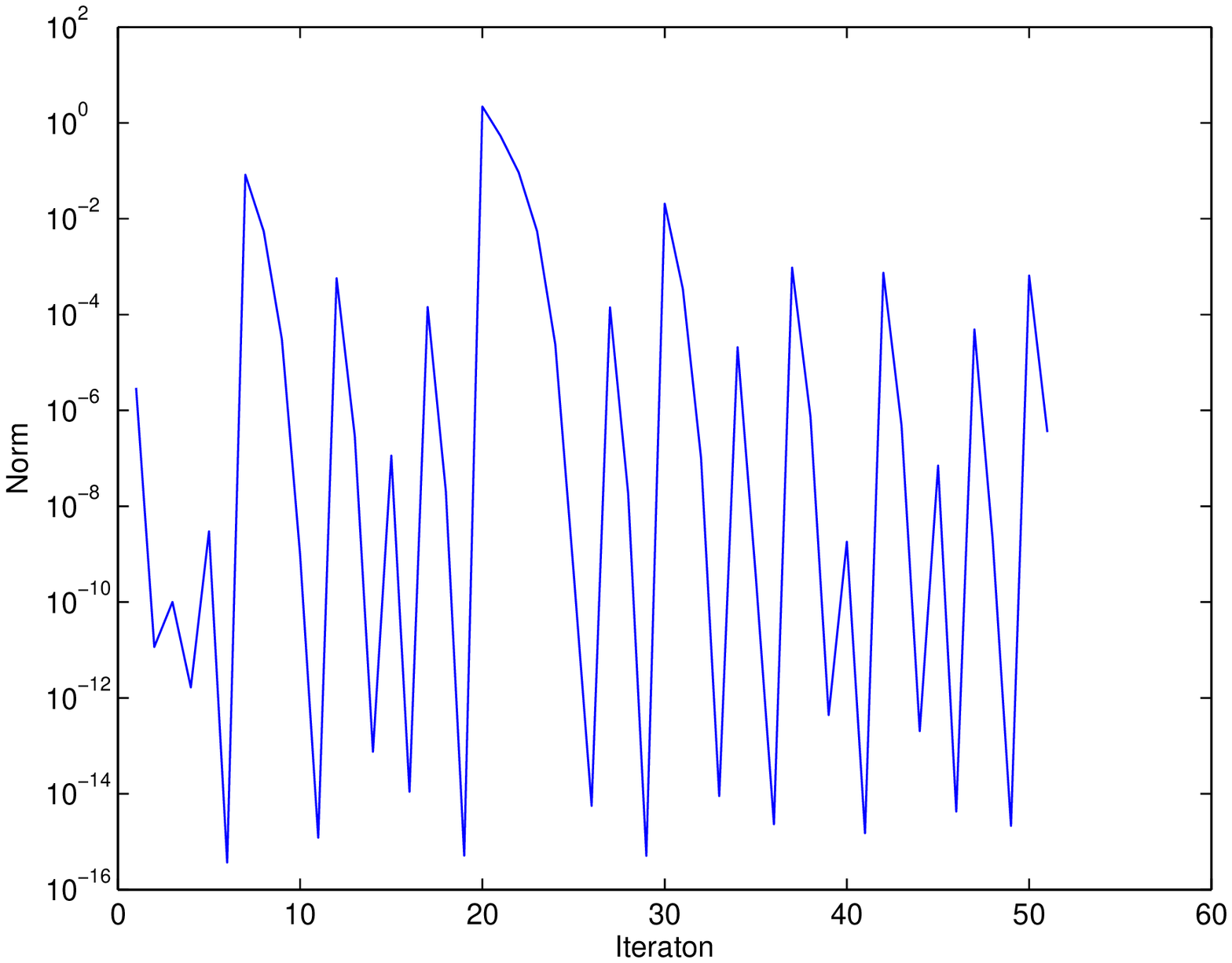}}
\subfigure[$K,\,XS,\,X\Omega$]{\includegraphics[width=0.5\textwidth]{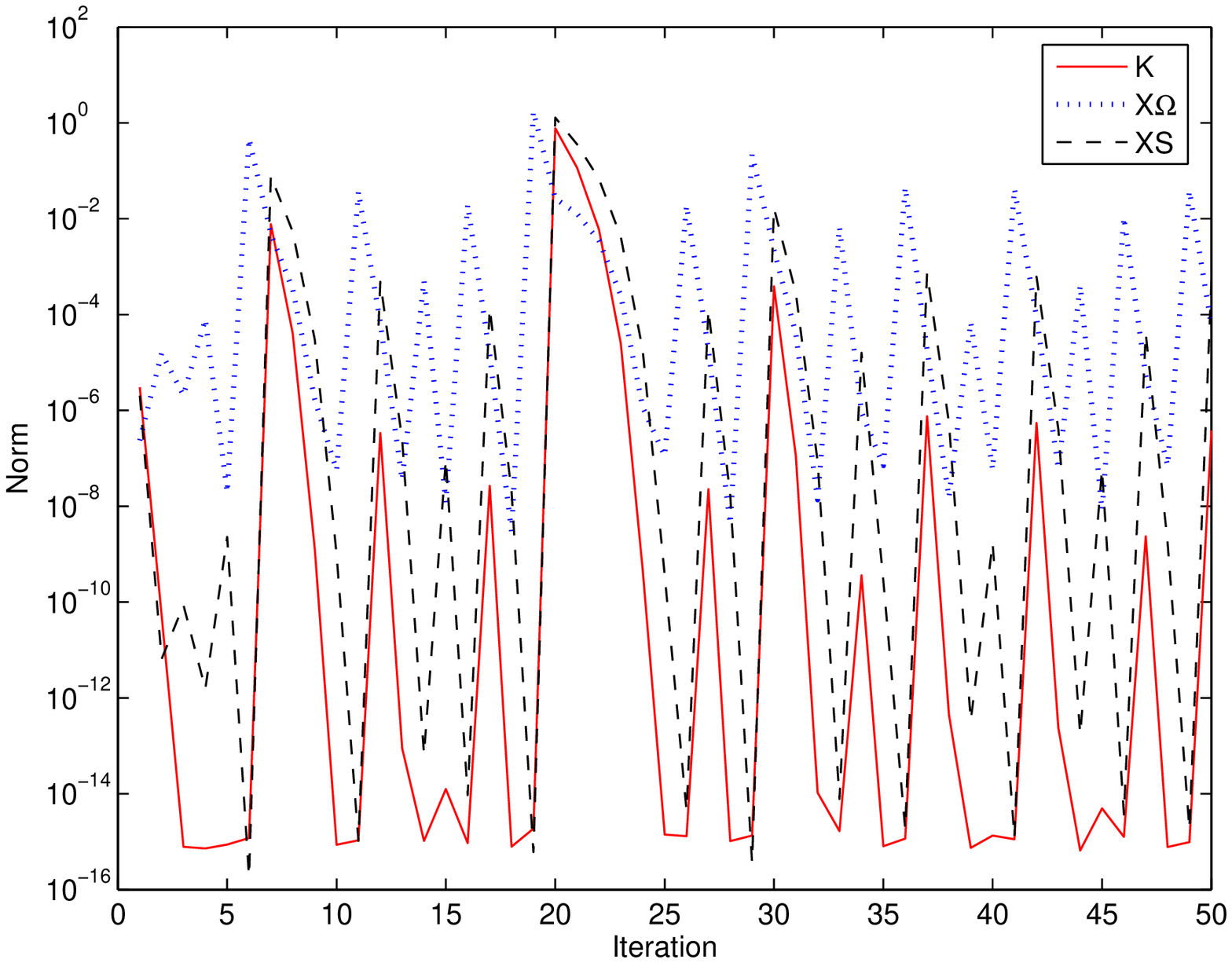}}
\caption{Plain Newton method}
\label{fig:plainN}
\end{figure}
In Figure~\ref{fig:plainN} (a), the norm of $F(X)$ is given for each iteration step. Close to the solution,
the system matrix gets singular and the algorithm deviates from the optimal point.
In Figure~\ref{fig:plainN} (b), we present the evolution of the norms of the three components
$K, X\Omega,$ and $XS$ of the update vector $Z,$
$$Z = X_\perp K + X \Omega + X S\;,$$
where $X_\perp$ is the orthogonal complement of $X,$ $\Omega$ is a
skew-symmetric matrix and $S$ is a symmetric matrix. We see that,
even when the $K$ and $XS$ component are very small, $X\Omega$ is quite large. 
This concords with the fact that the kernel of the Hessian at a stationary point $X$ is $\{X\Omega: \Omega^T=-\Omega\}$.

Next, we study the geometric Newton method, derived in Section 4. 
Again, we consider $n=6,\; p=3$ and 
a symmetric positive definite matrix $A\in\rr^{6\times 6}$ 
with uniformly distributed eigenvalues in the interval $[0,1]$. 
We perform $10^4$ experiments with a single matrix $A$ but with
different initial iterates $X_0\in\rr_\ast^{6\times 3}$ with random
entries, chosen from a normal distribution with zero mean and unit
variance. In Figure~\ref{fig:geomN}, we show the number of runs that
converged to each of the eigenspaces of $A$.
\begin{figure}[h]
\subfigure[Linear plot]{\includegraphics[width=0.5\textwidth]{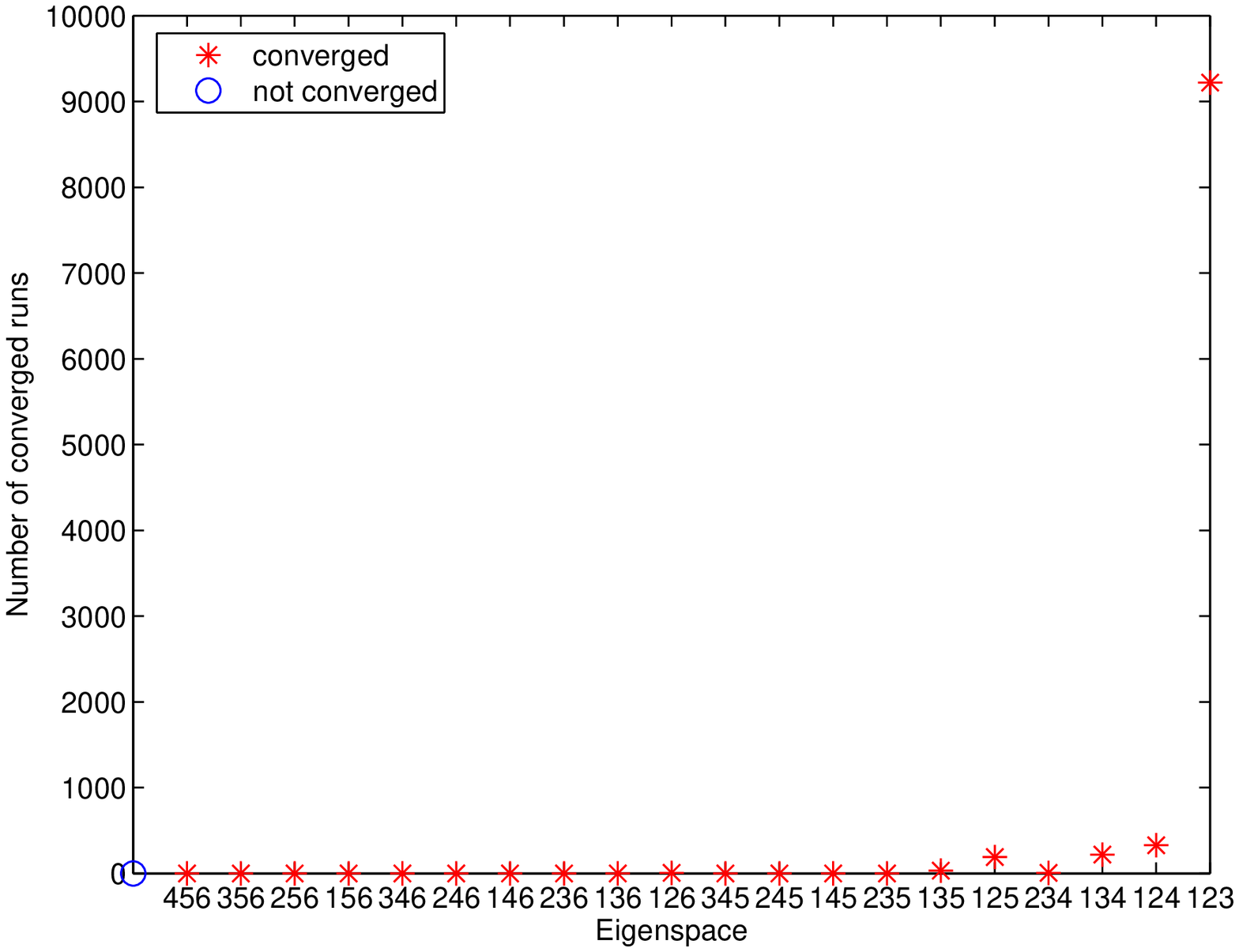}}
\subfigure[Logarithmic scale]{\includegraphics[width=0.5\textwidth]{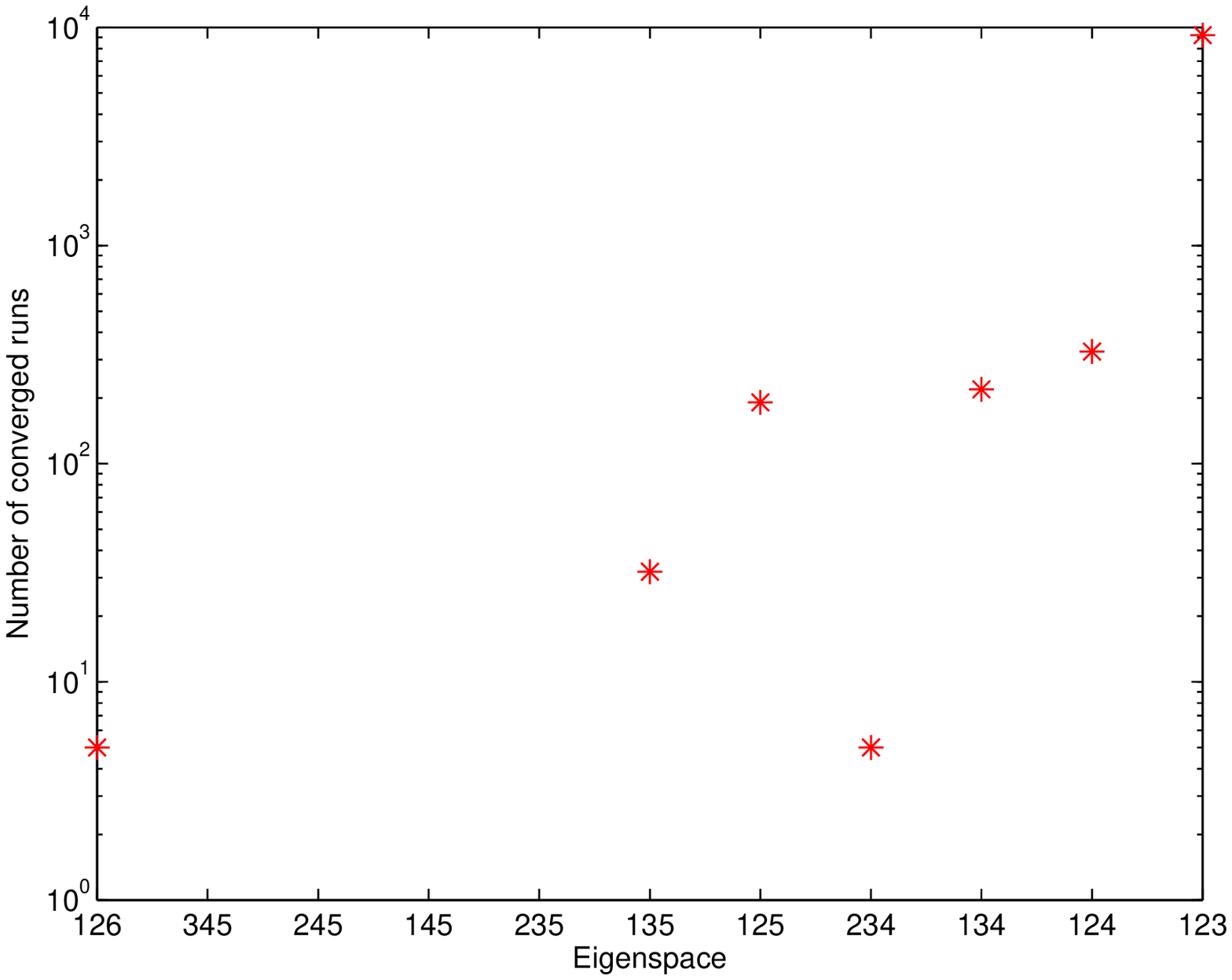}}
\caption{Geometric Newton method}
\label{fig:geomN}
\end{figure}
The dominant eigenspace is marked by ``123''. In general, ``$ijk$''
stands for the eigenspace spanned by the $i^{\mbox{\begin{footnotesize}th\end{footnotesize}}}$,
$j^{\mbox{\begin{footnotesize}th\end{footnotesize}}}$, and $k^{\mbox{\begin{footnotesize}th\end{footnotesize}}}$
eigenvectors \footnote{For convenience, we consider an eigenvalue decomposition, where
the eigenvalues of $A$ are put in descending order.}.
Let $W$ be the matrix of all eigenvectors.
We declare that the algorithm has converged to eigenspace ``$ijk$''
when the norms of the $i^{\mbox{\begin{footnotesize}th\end{footnotesize}}}$,
$j^{\mbox{\begin{footnotesize}th\end{footnotesize}}}$, and
$k^{\mbox{\begin{footnotesize}th\end{footnotesize}}}$ columns of the matrix $X^TW$
are all greater than $10^{-10}$ after 50 iterations and the norms of the rest of the columns are smaller than $10^{-10}$.
It appears that the basin of attraction of the dominant eigenspace
is the largest. In our experiment, all the runs have converged to
one of the 20 possible eigenspaces. In general, there may be cases
where the algorithm does not converge to any of the eigenspaces.
This may occur when the initial iterate $X_0$ is very close to the
boundary of one of the basins of attraction. However, these cases
are rare. Finally, the superlinear convergence rate of the
algorithm is illustrated in Figure~\ref{fig:superlin}.
\begin{figure}[h]
\begin{center}
{\includegraphics[width=0.5\textwidth]{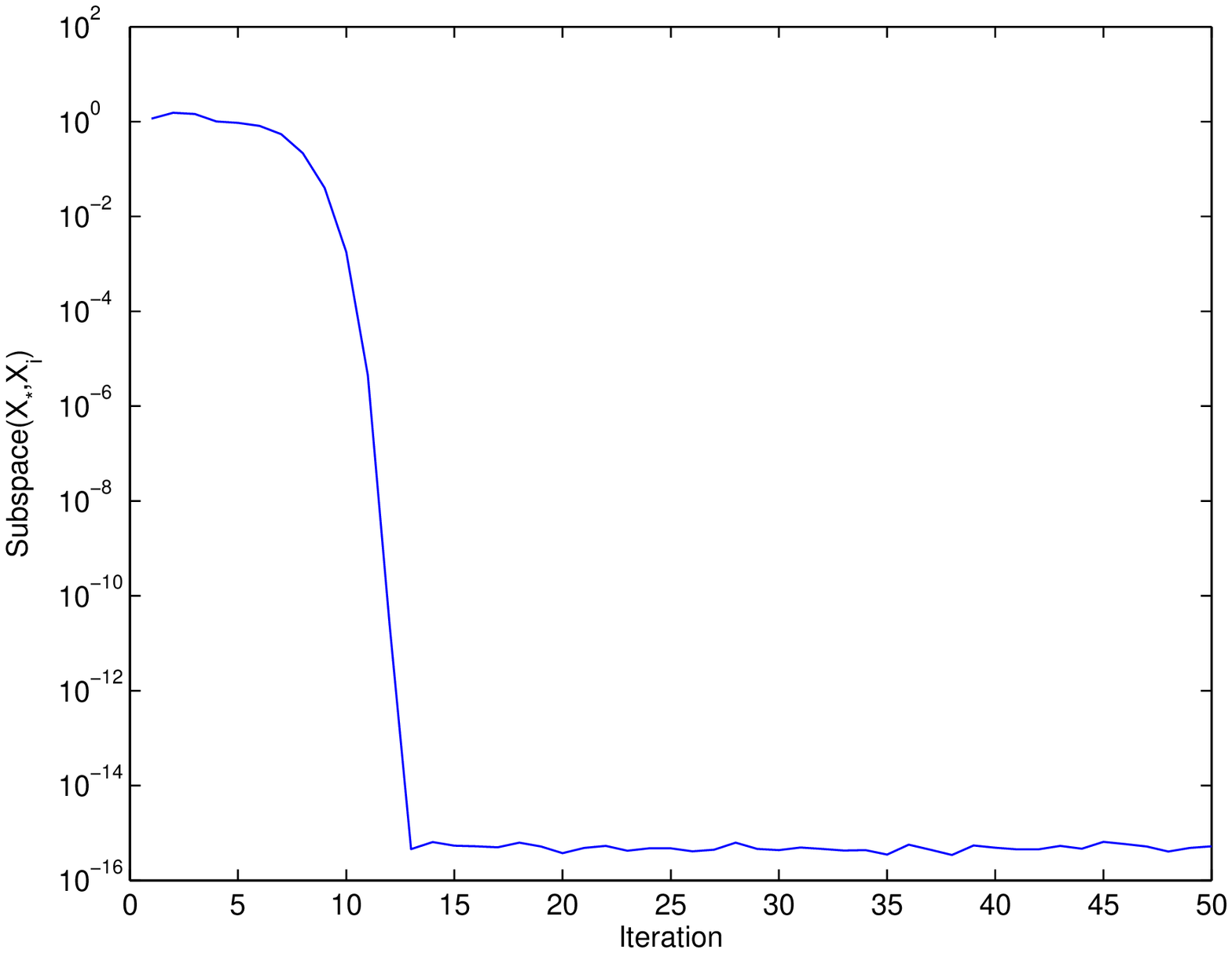}}
\caption{Superlinear convergence of the geometric Newton method}
\label{fig:superlin}
\end{center}
\end{figure}

\section{Conclusion}
\label{sec:conclusion}

We have investigated the use of Newton's method to compute
superlinearly the zeros of Oja's vector field~\eqref{eq:Oja-F}. Due to
a symmetry in the vector field by the action of the orthogonal group,
its zeros are never isolated, which causes the plain Newton method to
behave poorly. We have proposed a remedy that consists in
``quotienting out'' the symmetry. This led to the formulation of a
geometric Newton algorithm that seeks the zeros of a projection of
Oja's vector field. We have shown that the zeros of the projected
vector field are the same as the zeros of the original vector
field. Moreover, these zeros are nondegenerate. This means that by
quotienting out the action of the orthogonal group, we have removed
just enough symmetry to make the zeros nondegenerate. In view of the
nondegeneracy property, it follows directly from the convergence
theory of the abstract geometric Newton method that the resulting
algorithm converges locally superlinearly to the zeros of Oja's vector
field.

Invariant subspace computation has been and still is a very active
area of research. As a method for invariant subspace computation, it
is doubtful that the proposed algorithm can outperform the
state-of-the-art methods. In particular, the Grassmann-based
approach~\cite{EAS98,LE2002,AMS2004-01}, that can be thought of as
quotienting out the action of the whole general linear group instead of
the orthogonal group, leads to a Newton equation that lives in a
smaller subspace of $\rrnp$ and that can be solved in fewer
flops. When $n\gg p$, however, the number of operations to compute the
iterates is of the same order. Moreover, the Grassmann-based algorithm
does not possess the remarkable feature of naturally converging towards
orthonormal matrices, i.e., without having to resort to
orthogonalization steps such as Gram-Schmidt.

The problem of computing the zeros of Oja's vector field was also an
occasion for introducing the quotient manifold $\STpn/\O_p$, that seems
to have received little attention in the literature, in contrast to
the more famous Grassmann and Stiefel manifolds. In later work, we
will further analyze the geometry of this quotient manifold, which was
just touched upon in this paper, and we will show how it can be used
in the context of low-rank approximation problems.

\section*{Acknowledgements}

Fruitful discussions with Michel Journ\'ee, Rodolphe Sepulchre,
and Bart Vandereycken are gratefully acknowledged. 

\bibliographystyle{amsalpha_verycompact}
\bibliography{pabib}

\end{document}